\documentclass[12pt,reqno]{amsart}
\usepackage{amssymb,amsmath,amsthm,newlfont,enumerate}

\newtheorem{theorem}{Theorem}
\newtheorem{lemma}[theorem]{Lemma}
\newtheorem{definition}[theorem]{Definition}

\newcommand{\cF}{{\mathcal{F}_\varphi^2}(\CC)}

\newcommand{\Hol}{\mathrm{Hol}(\CC)}

\newcommand\CC{\mathbb{C}}

\newcommand\RR{\mathbb{R}}

\begin{document}

\title[Convergence of Lagrange interpolation]{Convergence of Lagrange interpolation series in the Fock spaces}
\author[Dumont]{Andr\'e Dumont}
\author[Kellay]{Karim Kellay}
\address{LATP\\CMI\\Universite Aix--Marseille \\
39, rue F. Joliot-Curie\\13453 Marseille\\France}

\email{dumont@cmi.univ-mrs.fr}
\address{IMB\\Universite Bordeaux I\\
351 cours de la Liberation\\33405 Talence \\France}
\email{Karim.Kellay@math.u-bordeaux1.fr}

\keywords{Fock spaces, Lagrange interpolation, convergence, summability methods}
\subjclass[2000]{Primary 30H05; Secondary 30D10, 30E05}

\begin{abstract} We study the uniqueness sets, the weak interpolation 
sets, and convergence of the Lagrange interpolation series in radial weighted Fock spaces.
\end{abstract}

\maketitle

\section{Introduction and main results.}
In this paper we study the weighted Fock spaces $\cF$
$$
\cF=\Bigl\{f\in \Hol \text{ : } \|f\|_\varphi^2=\int_\CC |f(z)|^2e^{-2\varphi(|z|)}\,dm(z)<\infty\Bigr\};
$$
here $dm$ is area measure and $\varphi$ is an increasing function defined on $[0,+\infty)$, $\lim_{x\to\infty}\varphi(x)=\infty$.   
We assume that $\varphi(z)=\varphi(|z|)$ is $C^2$ smooth and subharmonic on $\CC$, and set 
$$
\rho(z)=(\Delta \varphi(z))^{-1/2}.
$$
One more condition on $\varphi$ is that
for every fixed $C$,
$$
\rho(x+C\rho(x))\asymp \rho(x),\qquad 0<x<\infty.
$$
(In particular, this holds if $\rho'(x)=o(1)$, $x\to\infty$.)

Typical $\varphi$ are power functions
$$
\varphi(r)=r^a,\quad a>0.
$$
Then
$$
\rho(x)\asymp x^{1-a/2},\qquad x>1.
$$
Furthermore, if
$$
\varphi(r)=(\log r)^2, 
$$
then
$$
\rho(x)\asymp x,\qquad x>1.
$$

Given $z,w\in \CC$, we define 
$$
d_\rho(z,w)=\frac{|z-w|}{\min(\rho(z),\rho(w))}.
$$
We say that a subset $\Lambda$ of $\CC$  of is  $d_\rho$--separated  if
$$
\inf_{\lambda\neq\lambda^*}\{d_\rho(\lambda,\lambda^*),\; \lambda,\lambda^*\in \Lambda\}>0.
$$
\smallskip

\begin{definition}
Given $\gamma\in \RR$, we say that an entire function $S$ belongs 
to the class $\mathcal{S}_{\gamma}$ if 
\begin{enumerate}
\item[\text{$(1)$}] the zero set  $\Lambda$ of $S$ is 
$d_\rho$-separated, and 
\item[\text{$(2)$}] 
$$
|S(z)|\asymp {e^{\varphi(z)}} \frac{d(z,\Lambda)}{\rho(z)}\frac{1}{(1+|z|)^{\gamma}}, 
\qquad z\in \CC.
$$
\end{enumerate}
\end{definition}
\smallskip

For constructions of such functions in radial weighted
Fock spaces see, for example, \cite{BDK} and \cite{BL}.

In the standard Fock spaces ($\varphi(r)=r^2$) the classes 
$\mathcal{S}_{\gamma}$ were introduced by Lyubarskii in \cite{L}. They are analogs of the sine type functions for the Paley--Wiener space,
and their zero sets include rectangular lattices and their perturbations.

\begin{definition}
A set $\Lambda\subset\mathbb C$ is called {\it a weak interpolation set} 
for $\cF$ if for every $\lambda\in \Lambda$ there exists $f_\lambda\in \cF$ such that $f_\lambda(\lambda)=1$ and $f_\lambda|\Lambda\backslash\{\lambda\}=0$. 

A set $\Lambda\subset\mathbb C$ is called {\it a uniqueness set} 
for $\cF$ if $f\in \cF$ and  $f|\Lambda=0$ imply together that $f=0$.
\end{definition}

\begin{theorem}\label{uniqueness-minimal}  Let $S\in  \mathcal{S}_\gamma$, 
and denote by $\Lambda$ the zero set of $S$. Then 
\begin{enumerate}
\item[$(a)$] $\Lambda$ is a uniqueness set if and only if $\gamma\le 1$,
\item[$(b)$] $\Lambda$ is a weak interpolation set if and only if $\gamma>0$.
%\item[c)] $\Lambda$ est uniform\'ement minimal $\iff$ $0<\gamma<1$ et $\log r \ll \varphi(r)\ll (\log r)^2$
\end{enumerate}
\end{theorem}

Denote by $\Bbbk_z$ the reproducing kernel in the space $\cF$:
$$
\langle f, \Bbbk_z,\rangle_{\cF}=f(z),\qquad f\in\cF,\quad z\in\mathbb C.
$$

It is known \cite{S,IL,BL} that the space $\cF$ does not admit Riesz 
bases of the (normalized) reproducing kernels for regular 
$\varphi$, $\varphi(x)\gg(\log x)^2$. 
On the other hand, by Theorem~\ref{uniqueness-minimal}, for $0<\gamma\le 1$,
the family $\{\Bbbk_\lambda\}_{\lambda\in\Lambda}$ is a complete minimal family in $\cF$. Then the family $\{S/[S'(\lambda)(\cdot-\lambda)]\}_{\lambda\in\Lambda}$ is the biorthogonal system
and we associate to any $f\in\cF$ the formal (Lagrange interpolation) series
$$
f\sim \sum_{\lambda\in\Lambda}f(\lambda)\frac{S}{S'(\lambda)(\cdot-\lambda)}.
$$
It is natural to ask whether this formal series converges if we modify the norm of the space.

Denote by $\Lambda=\{\lambda_k\}$ the zero sequence of $S$
ordered in such a way that $|\lambda_k|\le|\lambda_{k+1}|$, 
$k\ge 1$. 
Following Lyubarskii \cite{L} and Lyubarskii--Seip \cite{LS} we 
obtain the following result:

\begin{theorem}\label{T: Convergence}
Let $0\le\beta\le 1$, $\gamma+\beta\in (1/2,1)$, and let 
$S\in \mathcal{S}_\gamma$.
Suppose that 
\begin{equation}
\label{x11}
r^{1-2\beta}=O(\rho(r)), \qquad r\to+\infty.
\end{equation} 
Then for every $f\in \cF$ we have
$$
\lim_{N\to \infty}
\Bigl\|f-S\sum_{k=1}^{N}\frac{f(\lambda_k)}{S'(\lambda_k)(\cdot-\lambda_k)}\Bigr\|_
{\varphi_\beta}=0, 
$$
where $\varphi_\beta(r)=\varphi(r)+\beta\log(1+r)$.
\end{theorem}

The result corresponding to $\beta=1/2$, $\varphi(r)=r^2$, 
$\rho(r)\asymp 1$ is contained in \cite[Theorem 10]{LS}. On the other hand, in the case $\varphi(r)=(\log r)^a$, $1<a\le 2$, $r>2$, 
$\rho(r)\asymp r$, the space $\cF$ contains Riesz bases of 
(normalized) reproducing kernels \cite{BL}.
Our theorem shows that in the case $\rho(r)\asymp r$, when 
$S\in \mathcal{S}_\gamma$, $\gamma\in (1/2,1)$,
the interpolation 
series converges already in $\cF$. Now, it is interesting  to find 
out how sharp is condition \eqref{x11} in Theorem~\ref{T: Convergence}.

\begin{theorem}\label{sharp}
Let $0<a\le 2$, $\varphi(r)=r^a$, $r>1$, 
$\rho(x)\asymp x^{1-a/2}$, $x>1$. If 
$0\le\beta<a/4$, $\gamma\in\mathbb R$, and  
$S\in \mathcal{S}_\gamma$, then there exists 
$f\in \cF$ such that
$$
\Bigl\|f-S\sum_{k=1}^{N}\frac{f(\lambda_k)}{S'(\lambda_k)(\cdot-\lambda_k)}\Bigr\|_
{\varphi_\beta}\not\to 0,\qquad N\to \infty. 
$$
\end{theorem}

Thus, for the 
power weights $\varphi(r)=r^a$, $0<a\le 2$, we need to modify the norm to get the convergence, and the critical value of $\beta$ is $a/4$.

The notation $A\lesssim B$ means that there is a constant $C$ independent of the
relevant variables such that $A \leq CB$. We write   $A\asymp B$  if both $A\lesssim B$ and $B\lesssim A$.

\subsection*{Acknowledgements} The authors are grateful to Alexander Borichev for  very helpful discussions and comments. 

\section{Proofs.}

\subsection{Proof of Theorem \ref{uniqueness-minimal}} (a) If 
$\gamma>1$, then $\mathcal{S}_\gamma\in\cF$ and $S|\Lambda=0$. Hence, 
$\Lambda$ is not a uniqueness set. 

If $\gamma\leq 1$, then $\mathcal{S}_\gamma \cap \cF=\emptyset$. Suppose that there exists $g\in \cF$ such that  $g|\Lambda=0$. 
Then $g=FS$ for an entire function $F$, and
\begin{equation}
\label{integralc}
\int_\CC |F(w)|^2|S(w)|^2e^{-2\varphi(w)}\,dm(w)<\infty.
\end{equation}

Given $\Omega\subset \CC$, denote  
$$
\mathcal{I}[\Omega]=\int_{\Omega} |F(w)|^2\frac{d^2(w,\Lambda)}{(1+|w|)^{2\gamma}\rho^2(w)}\,dm(w).
$$
By \eqref{integralc}, 
$$
\mathcal{I}[\CC]<\infty.
$$
Denote by $D(z,r)$ the disc of radius $r$ centered at $z$. Let 
$$
\Omega_\varepsilon=\bigcup_{\lambda\in \Lambda}D(\lambda,\varepsilon\rho(\lambda)),
$$ 
where $\varepsilon$ is such that the discs 
$D(\lambda,2\varepsilon\rho(\lambda))$ are pairwise disjoint.
We have  
$$
\mathcal{I}[\CC]=\mathcal{I}[\CC\backslash \Omega_{2\varepsilon}]+\sum_{\lambda\in \Lambda}\mathcal{I}[D(\lambda,2\varepsilon\rho(\lambda))\backslash D(\lambda,\varepsilon\rho(\lambda))]+ \mathcal{I}[D(\lambda,\varepsilon\rho(\lambda))].
$$
It is clear that
$$
\mathcal{I}[{\CC\backslash \Omega_{2\varepsilon}}] \geq c_1\int_{\CC\backslash \Omega_{2\varepsilon}}  \frac{|F(w)|^2}{(1+|w|)^{2\gamma}}\,dm(w).
$$
On the other hand,
$$
\int_{ D(\lambda,\varepsilon\rho(\lambda))}  |F(w)|^2\,dm(w)\leq c_2\int_{ D(\lambda,2\varepsilon\rho(\lambda))\backslash D(\lambda,\varepsilon\rho(\lambda))} |F(w)|^2\,dm(w),
$$
and, hence, 
$$
\mathcal{I}[D(\lambda,2\varepsilon\rho(\lambda))\backslash D(\lambda,\varepsilon\rho(\lambda))]\geq c_3 \mathcal{I}[D(\lambda,\varepsilon\rho(\lambda))].
$$
Therefore 
$$
\int_{\CC} \frac{|F(w)|^2}{(1+|w|)^{2\gamma}}\,dm(w)<\infty, 
$$
the function $F$ is constant, and $g=cS$. Since $\mathcal{S}_\gamma \cap \cF=\emptyset$, 
we get a contradiction. Statement (a) is proved.

(b) Let $\gamma>0$. Set 
$$
f_\lambda(z)=\frac{{S}(z)}{{S'}({\lambda})(z-{\lambda})},\qquad 
\lambda\in\Lambda.
$$
It is obvious that $f_\lambda\in \cF$, $f_\lambda|\Lambda\backslash\{\lambda\}=0$ and $f_\lambda(\lambda)=1$. 
Hence $\Lambda$ is a weak interpolation set. If $\gamma\leq 0$, $\lambda\in \Lambda$, then by (a), $\Lambda\backslash\{\lambda\}$ is a uniqueness set for $\cF$.
Therefore, $\Lambda$ is not a weak interpolation set for $\cF$.
\hfill $\Box$

\subsection{Proof of Theorem \ref{T: Convergence}} 
We follow the scheme of proof proposed in \cite{L,LS} and  concentrate mainly on the places where the proofs differ. We need some auxiliary notions and lemmas. The proof of the first lemma is the same as in \cite[Lemma 4.1]{BDK}.

\begin{lemma}\label{L:inequalitysubha} For every $\delta>0$, there exists 
$C>0$ such that for functions $f$ holomorphic in $D(z,\delta\rho(z))$  
we have
$$
|f(z)|^2e^{-2\varphi(z)}\le  
\frac{C}{\rho(z)^2}\int_{D(z,\delta\rho(z))}|f(w)|^2e^{-2\varphi(w)}\,dm(w).
$$
\end{lemma}

\begin{definition} A simple closed curve $\gamma=\{r(\theta)e^{-i\theta}, \, \theta\in [0,2\pi]\}$ is called  K--bounded 
if $r$ is $C^1$-smooth and $2\pi$-periodic on the real line and
$|r'(\theta)|\le K$, $\theta\in\mathbb R$.
\end{definition}

Let $\gamma \in \mathbb R$, let $S\in \mathcal{S}_\gamma$, and let 
$\Lambda=\{\lambda_k\}$ be the zero set of $S$ ordered in such a 
way that $|\lambda_k|\le|\lambda_{k+1}|$, $k\ge 1$. We can 
construct a sequence of numbers  $R_N\to \infty$ and a sequence of 
contours $\Gamma_N$ such that 
\begin{enumerate}
\item $\Gamma_N=R_N\gamma_N$, where $\gamma_N$ are $K$-bounded, with $K>0$ independent of $N$.
\item $d_\rho(\Lambda,\Gamma_N)\geq \varepsilon $ for some $\varepsilon>0$ independent of $N$.
\item $\{\lambda_k\}_1^N$ lie inside $\Gamma_N$ and  $\{\lambda_k\}_{N+1}^{\infty}$ lie outside $\Gamma_N$.
\item $\Gamma_N\subset\{z:R_N-\rho(R_N)<|z|<R_N+\rho(R_N)\}$.
\end{enumerate}

Indeed, for some $0<\varepsilon<1$ the discs 
$D_k=D(\lambda_k,\varepsilon\rho(\lambda_k))$ are disjoint. 
For some $\delta=\delta(\varepsilon)>0$ we have 
$$
\varepsilon\rho(\lambda_k)>4\delta\rho(\lambda_N),\qquad
\bigl||\lambda_k|-|\lambda_N|\bigr|<\delta\rho(\lambda_N).
$$
Fix $\psi\in C^\infty_0[-1,1]$, $0<\psi<1$, such that 
$\psi>1/2$ on $[-1/2,1/2]$. 

Put $\Xi=\{k:\bigl||\lambda_k|-|\lambda_N|\bigr|<\frac14\delta\rho(\lambda_N)\}$, denote $\lambda_k=r_ke^{i\theta_k}$, $k\in\Xi$, and set 
$$
r(\theta)=1+\sum_{k\in\Xi}s_k\frac{\delta\rho(\lambda_N)}
{|\lambda_N|}\psi\bigl(\frac
{|\lambda_N|}{\delta\rho(\lambda_N)}(\theta-\theta_k)\bigr),
$$
where $s_k=1$, $k\le N$, $s_k=-1$, $k>N$.

Finally, set 
$$
\gamma_N=\{r(\theta)e^{i\theta}, \, \theta\in [0,2\pi]\}.
$$

\begin{lemma}\label{L: principal}
$$
R_N\rho(R_N)\int_{\gamma_N}|f(R_N\zeta)|^2e^{-2\varphi(R_N\zeta)}|d\zeta|\to 0,\qquad N\to \infty.
$$
\end{lemma}

\begin{proof} Set $C_N=\cup_{\zeta\in \gamma_N} D(R_N\zeta, \rho(R_N\zeta))$. Since $\rho(R_N\zeta)\asymp \rho(R_N)$, $\zeta\in\gamma_N$, by Lemma \ref{L:inequalitysubha} we have 
\begin{multline*}
R_N\rho(R_N)\int_{\gamma_N}|f(R_N\zeta)|^2e^{-2\varphi(R_N\zeta)}|d\zeta|\\
\lesssim R_N\rho(R_N)\int_{\gamma_N} \Big[\frac{1}{\rho(R_N\zeta)^2}\int_{D(R_N\zeta, \rho(R_N\zeta))}|f(w)|^2e^{-2\varphi(w)}\,dm(w)\Big] |d\zeta|\\
\asymp \frac{R_N}{\rho(R_N)}\int_{C_N}|f(w)|^2e^{-2\varphi(w)}\Big(\int_{\gamma_N}\chi_{D(R_N\zeta, \rho(R_N\zeta))}(w) |d\zeta|\Big)\,dm(w)\\
\lesssim\int_{C_N}|f(w)|^2e^{-2\varphi(w)}\,dm(w) \to 0,\qquad N\to \infty.
\end{multline*}
\end{proof}

\subsection*{Proof of Theorem \ref{T: Convergence}} Let $\chi_N(z)=1$, if  $z$ lies inside  $\Gamma_N$ and $0$ otherwise. Put 
$$
\Sigma_N(z,f)=S(z)\sum_{k=1}^{N}\frac{f(\lambda_k)}{S'(\lambda_k)(z-\lambda_k)},
$$
and set 
$$
I_N(z,f)=\frac{1}{2\pi i}\int_{\Gamma_N}\frac{f(\zeta)}{S(\zeta)(z-\zeta)}d\zeta
$$
The Cauchy formula gives us that 
$$
I_N(z,f)=\sum_{k=1}^{N}\frac{f(\lambda_k)}{S'(\lambda_k)(z-\lambda_k)}-\chi_N(z)\frac{f(z)}{S(z)},\qquad z\notin \Gamma_N.
$$
Hence, 
$$ 
\Sigma_N(z,f)-{f(z)} =S(z)   I_N(z,f) +(\chi_N(z)-1){f(z)},
$$
and it remains only to prove that
$$
\|S I_N(\cdot,f)\|_{{\varphi_\beta}}\to 0,\qquad N\to\infty.
$$
Let $\omega$ be a Lebesgue measurable function such that 
\begin{equation}
\label{Omega}
\int_{0}^{\infty}\int_{0}^{2\pi} |\omega(re^{it})|^2e^{-2\varphi(r)}(1+r)^{-2\beta}rdr dt \le 1, 
\end{equation}
and let 
$$
J_N(f,\omega)=\int_{0}^{\infty}\int_{0}^{2\pi} \omega(re^{it})S(re^{it})I_N(re^{it},f)e^{-2\varphi(r)}(1+r)^{-2\beta} rdr dt.
$$
It remains to show that
$$
\sup |J_N(f,\omega)|\to 0, \qquad N\to \infty,
$$
where the supremum is taken over all $\omega$ satisfying \eqref{Omega}. 

We have
\begin{gather*}
2\pi iJ_N(f,\omega)=\int_{\Gamma_N}\frac{f(\zeta)}{S(\zeta)}\int_\CC\frac{\omega(z)S(z)}{z-\zeta}e^{-2\varphi(z)}(1+|z|)^{-2\beta}{dm(z)}d\zeta\\
=\int_{\Gamma_N}\frac{f(\zeta)}{S(\zeta)}\int_\CC\frac{\phi(z)}{z-\zeta}(1+|z|)^{-\beta-\gamma}{dm(z)}d\zeta, 
\end{gather*}
where
$$
\phi(z)=[\omega(z)e^{-\varphi(z)}(1+|z|)^{-\beta}][S(z)e^{-\varphi(z)}(1+|z|)^{\gamma}].
$$ 
Note that 
$$
\int_\CC|\phi(z)|^2\,dm(z)\leq C.
$$
Set $\psi(z)=R_N\phi(R_Nz)$. We have 
$$
\int_\CC|\psi(z)|^2\,dm(z)\le C.
$$ 
Changing the variables $z=R_Nw$ and $\zeta=R_N\eta$, we get 
$$
2\pi iJ_N(f,\omega)=
R_N\int_{\gamma_N}\frac{f(R_N\eta)}{S(R_N\eta)}\int_\CC\frac{\psi(w)}{w-\eta}(1+R_N|w|)^{-\beta-\gamma}{dm(w)}d\eta.
$$
Consider the operators 
$$
T_N(\psi)(\eta)=\int_\CC\frac{\psi(w)}{w-\eta}|w|^{-\beta-\gamma}{dm(w)}, \qquad \psi \in L^{2}(\CC,dm(w)).
$$
Since $\gamma+\beta\in (1/2,1)$, by \cite[Lemma~13]{LS}, the operators $T_N$ are bounded from $L^{2}(\CC,dm(w))$ into $L^2(\gamma_N)$ and  
$$
\sup_N \|T_N\|<\infty.
$$ 
Hence, by Lemma~\ref{L: principal} and by the property 
$r^{1-2\beta}=O(\rho(r))$, $r\to\infty$, we get 
\begin{multline*}
J_N(f,\omega)\lesssim  R_{N}^{1-\beta-\gamma}\Big|\int_{\gamma_N}\frac{f(R_N\eta)}{S(R_N\eta)}T_N(\psi)(\eta)d\eta\Big|\\
\lesssim R_{N}^{1-\beta}\int_{\gamma_N}{|f(R_N\eta)|}{e^{-\varphi(R_N\eta)}}|T_N(\psi)(\eta)|\,|d\eta|\\
\lesssim\Big( R_{N}\rho(R_N)\!\int_{\gamma_N}|f(R_N\eta)|^2{e^{-2\varphi(R_N\eta)}}|d\eta|\Big)^{1/2}
\\
\times\|T_N(\psi)\|_{L^{2}(\CC,dm(w))} \to 0,
\qquad N\to\infty.
\end{multline*}
This completes the proof. 
\hfill $\Box$

\subsection{Proof of Theorem \ref{sharp}}
It suffices to find $f\in \cF$ and a sequence $N_k$ such that 
(in the notations of the proof of Theorem~\ref{T: Convergence})
\begin{equation}
A_k=\Bigl\|S\chi_{_{N_k}}\int_{\Gamma_{N_k}}\frac{f(\zeta)}
{S(\zeta)(\cdot-\zeta)}\,d\zeta\Bigr\|_{\varphi_\beta}\not\to 0,\qquad k\to\infty.
\label{x14}
\end{equation}
We follow the method of the proof of \cite[Theorem~11]{LS}.
Let us write down the Taylor series of $S$:
$$
S(z)=\sum_{n\ge 0}s_nz^n.
$$
Since $S\in\mathcal S_\gamma$, we have
$$
|s_n|\le c\exp\Bigl(-\frac na\ln\frac n{ae}-\frac \gamma{a}\ln n\Bigr),\qquad n>0.
$$
Choose $0<\varepsilon<\frac a2-2\beta$. Given $R>0$ consider
$$
S_R=\sum_{|n-aR^a|<R^{{\scriptstyle  
\frac a2}+\varepsilon}}s_nz^n.
$$
Then for every $n$ we have 
$$
|S(z)-S_R(z)|e^{-|z|^a}=O(|z|^{-n}),\qquad \bigl||z|-R\bigr|<\rho(R),\qquad R\to\infty.
$$

Next we use that for some $c>0$ independent of $n$,
$$
\int_0^\infty r^{2n+1}e^{-2r^a}\,dr\le c
\int_{|r-({\scriptstyle \frac na})^{1/a}|<n^{(1/a)-(1/2)}} r^{2n+1}e^{-2r^a}\,dr.
$$ 
Therefore,
\begin{gather*}
\|S_R\|^2_\varphi=\sum_{|n-aR^a|<R^{{\scriptstyle\frac a2}+\varepsilon}}
\pi |s_n|^2\int_0^\infty r^{2n+1}e^{-2r^a}\,dr
\\
\le \sum_{|n-aR^a|<R^{{\scriptstyle\frac a2}+\varepsilon}}
c |s_n|^2\int_{|r-({\scriptstyle\frac na})^{1/a}|<n^{(1/a)-(1/2)}} r^{2n+1}e^{-2r^a}\,dr
\\
\le\sum_{|n-aR^a|<R^{{\scriptstyle \frac a2}+\varepsilon}}
c |s_n|^2\int_{|r-R|<c_1R^{1-{\scriptstyle\frac a2}+\varepsilon}} r^{2n+1}e^{-2r^a}\,dr
\\
\le\sum_{n\ge 0}
c |s_n|^2\int_{|r-R|<c_1R^{1-{\scriptstyle\frac a2}+\varepsilon}} r^{2n+1}e^{-2r^a}\,dr
\\
=c\int_{|r-R|<c_1R^{1-{\scriptstyle \frac a2}+\varepsilon}} \sum_{n\ge 0}
 |s_n|^2r^{2n+1}e^{-2r^a}\,dr
\\
=c\int_{\bigl||z|-R\bigr|<c_1R^{1-{\scriptstyle\frac a2}+\varepsilon}} |S(z)|^2
e^{-2|z|^a}\,dm(z)\le c_2R^{2-{\scriptstyle\frac a2}+\varepsilon-2\gamma}.
\end{gather*}
Fix $\varkappa$ such that
$$
1-\frac a4+\frac{\varepsilon}2-\gamma<\varkappa<1-\beta-\gamma.
$$

Choose a sequence $N_k$, $k\ge 1$, such that for 
$R_k=|\lambda_{N_k}|$ we have $R_{k+1}>2R_k$, $k\ge 1$, and
$$
\Bigl|e^{-|z|^a}\sum_{m\not=k}S_{R_m}(z)R_m^{-\varkappa}\Bigr|\le \frac 1{|z|^{\gamma+1}},\qquad \bigl||z|-R_k\bigr|<\rho(R_k),\quad k\ge 1.
$$

Set
$$
f=\sum_{k\ge 1}S_{R_k}R_m^{-\varkappa}.
$$
Then $f\in\cF$, and
$$
\frac fS=R_k^{-\varkappa}+O(R_k^{-1-\varkappa}) \quad \text{on\ \ }\Gamma_{N_k},\quad k\to\infty.
$$
Hence,
$$
\Bigl|S(z)\int_{\Gamma_{N_k}}\frac{f(\zeta)}
{S(\zeta)(z-\zeta)}\,d\zeta\Bigr|\ge cR_k^{-\varkappa}\frac{e^{|z|^a}}{(1+|z|)^\gamma},\qquad |z|<\frac {R_k}2,
$$
and finally 
$$
A_k\ge cR_k^{-\varkappa}\Bigl(\int_0^{R_k/2}\frac {r^{1-2\beta}\,dr}{(1+r)^{2\gamma}}\Bigr)^{1/2}\to\infty, \qquad k\to\infty.  
$$
This proves \eqref{x14} and thus completes the proof of the theorem.
\hfill $\Box$

\end{document}